\documentclass[14pt,reqno, a4paper]{amsart}

\usepackage[utf8]{inputenc}
\usepackage[T1]{fontenc}%polish characters
\usepackage{amssymb}
\usepackage{amscd}
\usepackage{amsxtra}
\usepackage{amsmath}
\usepackage{appendix}
\usepackage{amsfonts}
\usepackage{amsthm}
\usepackage{calligra}
\usepackage{comment}
\usepackage{csquotes}
\usepackage{dsfont}
\usepackage{graphicx}
\usepackage{enumitem}
\usepackage{placeins}
\usepackage{mathscinet}
\usepackage{mathrsfs}
\usepackage[english]{babel} %important package used with biblatex, usually you need to reset your compiler to run babel instead of bibtex
\usepackage{hyperref}  %linkable references and chapters
\usepackage[nameinlink]{cleveref}
\usepackage{mathtools}
\usepackage{stackrel,amssymb}
\usepackage{tikz-cd}
\usepackage[foot]{amsaddr}% email address in footer
\usepackage[backend=biber,style=alphabetic,sorting=nyt,maxbibnames=5,doi=false,isbn=false,url=false]{biblatex} %biblatex is great, use it with babel and csquotes, the parameters are the settings, alphabetic gives you the name+year, refsection means 1 bib per section

\DeclareLabelalphaTemplate[]{
  \labelelement{
    \field[final]{shorthand}
    \field{label}
    \field[strwidth=3,strside=left,ifnames=1]{labelname}
    \field[strwidth=1,strside=left]{labelname}
  }
}
\DeclareFieldFormat{extraalpha}{#1}
\newtheorem{theorem}{Theorem}[subsection]

\numberwithin{theorem}{subsection}

\newtheorem{lemma}[theorem]{Lemma}

\newtheorem{corollary}[theorem]{Corollary}
\newtheorem{proposition}[theorem]{Proposition}

\newtheorem*{varthm}{}

\theoremstyle{remark}
\newtheorem{remark}[theorem]{Remark}

\theoremstyle{definition}
\newtheorem{definition}[theorem]{Definition}

\makeatletter
\def\thmhead@plain#1#2#3{%
  \thmname{#1}\thmnumber{\@ifnotempty{#1}{ }\@upn{#2}}%
  \thmnote{ {\the\thm@notefont#3}}}
\let\thmhead\thmhead@plain
\makeatother

\newcommand{\proj}{\text{pr}}
\newcommand{\Osh}{\mathcal{O}}

\newcommand{\Spec}{{\rm Spec}}

\newcommand{\leftp}{\left(}
\newcommand{\rightp}{\right)}

\newcommand{\supp}{{\normalfont\text{supp}}}

\newcommand{\Orb}[1]{{O(#1)}}

\newcommand{\ellm}[1]{{{\ell^{S^*}}(#1)}}
\newcommand{\XF}[1]{{\overline{X}^{S^*}(#1)}}
\newcommand{\Xo}[1]{{\overline{X}(#1)}}
\newcommand{\OrbF}[1]{{{\rm Orb}_F(#1)}}

\DeclareMathOperator{\sheafhom}{\mathscr{H}\text{\kern -3pt {\textit{om}}}}
\DeclareMathOperator{\sheafdif}{\mathscr{D}\text{\kern -1pt {\textit{if}}}}

\usepackage{etoolbox}
\makeatletter
\patchcmd{\@maketitle}
  {\ifx\@empty\@dedicatory}
  {\ifx\@empty\@date \else {\vskip3ex \centering\footnotesize\@date\par\vskip1ex}\fi
   \ifx\@empty\@dedicatory}
  {}{}
\patchcmd{\@adminfootnotes}
  {\ifx\@empty\@date\else \@footnotetext{\@setdate}\fi}
  {}{}{}
\makeatother

\title[Cohomology  of compactified Deligne--Lusztig varieties]{Geometry and cohomology of compactified Deligne--Lusztig varieties}
\author{Yingying Wang}
\begin{document}
\maketitle
\begin{abstract}

For connected reductive groups together with a Frobenius root $F$, we show that
the cohomology of the structure sheaf and respectively the canonical sheaf for compactified Deligne--Lusztig varieties associated to an element  in the free monoid generated by the simple reflections is isomorphic to that of a minimal length element in an $F$-conjugacy class in the Weyl group.\end{abstract}
\tableofcontents
\section{Introduction}\label{intro}

\subsection{Background}\label{secbackground}
Let $q$ be a power of a prime $p$, $k_0=\mathbb F_{q^d}$ for a positive integer $d$, and $k$ an algebraic closure of $\mathbb{F}_p$ containing $k_0$. Let $G_0$ be a connected reductive $k_0$-group with $q^d$-Frobenius $F_0:G_0\to G_0$, and let $G$  be the base change of $G_0$ to $k$ together with an endomorphism $F:G\to G$ such that $F^d=F_0\otimes_{k_0}{\rm Id}$. 
 The group of fixed points $G^F$ 
 is known as a finite group of Lie type. Examples of such groups include  ${\rm GL}_n(\mathbb{F}_{q})$ and the subgroup of unitary matrices in ${\rm GL}_n(\mathbb{F}_{q^2})$.

Let $T\subset B$ be an $F$-stable maximal torus and an $F$-stable Borel subgroup of $G$.
Denote by $W:=N(T)/T$ the corresponding Weyl group and let $S\subseteq W$ be the set of simple reflections.
For any $w\in W$, the associated Deligne--Lusztig variety is denoted by $X(w)$. Throughout the introduction, we fix a reduced expression of $w$ and denote by $\overline{X}(w)$ the corresponding  smooth compactification of $X(w)$ (\Cref{2.1definition}).  

 In \cite{DL}, Deligne and Lusztig considered the $\ell$-adic cohomology of $X(w)$ and certain Galois coverings of $X(w)$ to study the irreducible
  $G^F$-representations over fields of characteristic zero. 
The de Rham and mod $p$ cohomology of $X(w)$ and $\overline{X}(w)$ are less understood than their $\ell$-adic counterpart. 
However, for ${\rm GL}_2(\mathbb{F}_q)$ and $w$ the Coxeter element, the de Rham cohomology of $\overline{X}(w)$ is not mysterious because $\overline{X}(w)$ is isomorphic to the projective line.

The simplicity in the situation of $\overline{X}(w)$ for ${\rm GL}_2(\mathbb{F}_q)$ does not carry over to higher dimensions. The smooth compactification $\overline{X}(w)$ is constructed in loc. cit. as the transversal intersection between the graph of  $F$ and the desingularisation of the Schubert variety associated to $w$ constructed by Demazure and Hansen. Unlike 
 these desingularisations of Schubert varieties,  $\overline{X}(w)$ is not naturally an iterated $\mathbb{P}^1$-fibration. Thus the cohomology groups of  $\overline{X}(w)$ cannot be computed in the same way.

In \cite{Wang22},  we described the cohomology of the structure sheaf of  $\overline{X}(w)$ for any $w\in W$ in the case of  ${\rm GL}_n(\mathbb{F}_q)$. Our main theorem showed that  
the cohomology of the structure sheaf of  $\overline{X}(w)$ is always isomorphic to that of $\overline{X}(c)$, where $c$ is a Coxeter element in a parabolic subgroup of $W$ with $\supp(c)=\supp(w):=\{s\in S \ \vert \ s\leq w\}$. 
As a consequence of this  theorem, we were able to describe the compactly supported integral $p$-adic cohomology of any Deligne--Lusztig variety $X(w)$ for  ${\rm GL}_n(\mathbb{F}_q)$. 

\subsection{Main results}

The aim of this paper is to extend the main theorem of \cite{Wang22} to all pairs $(G,F)$ as in \textsection 1.1. 
After fixing a suitable pair $T\subset B$ of $F$-stable maximal torus and  Borel subgroup, we may assume that $F(S)=S$ without loss of generality \cite[\textsection 4.2]{DigneMichelBook}. 
In fact, when  $G$ is semisimple, $F$ would be the Frobenius endomorphism (i.e. $d=1$) except for the cases for Suzuki and Ree groups, for which $d=2$ and additional assumptions on $q$ are required
(\Cref{secsuzukiree}).

In our setting, we encounter the minimal length elements in the $F$-conjugacy classes of $W$ (\Cref{subsecFconjugacyclassesandFcyclicshifts}), which are special elements in $W$ 
with nice properties \cite[Theorem 2.6]{GeckKimPfeifferMinlength}, \cite[Theorem 3.1]{HeNieminlength}.
When $G^F$ is ${\rm GL}_n(\mathbb{F}_q)$, they are precisely the Coxeter elements in parabolic subgroups of $W$. In general, Coxeter elements of $W$ (in the sense of  \Cref{conventionforcoxeter}) are examples of minimal length elements in $F$-conjugacy classes.   
 
\begin{varthm}[\textbf{\Cref{mainthm}}]\label{introtheorem}
For every $w\in W$
there exists 
an $F$-conjugacy class $C$ in $W$ and a minimal length element $x_0\in C$ such that  for all $i\geq 0$, there are $G^F$-equivariant isomorphisms
	 	\begin{equation*}
		H^i\left(\Xo{w},\mathcal{F}_{\Xo{w}}\right)\overset{\sim}{\longrightarrow} H^i\left(\Xo{x_0},\mathcal{F}_{\Xo{x_0}}\right),
	\end{equation*}
where $\mathcal{F}\in \{\mathcal{O}, \omega\}$ is either the structure sheaf or the canonical sheaf. 
	 \end{varthm}

Note that the  $F$-conjugacy class $C$ in the theorem may not contain $w$. However, for any given $w$, one can determine $x_0$ and $C$ explicitly by following our 
 recipe given in the proof of \Cref{mainthm}. We will also show that the cohomology groups in the theorem do not depend on the choice of reduced expressions of $w$ and $x_0$ and their corresponding smooth compactifications. 

On the other hand, smooth compactifications of $X(w)$ given by different reduced expressions of $w$ are not necessarily isomorphic. 
 To avoid this ambiguity, we work with 
 compactified 
Deligne--Lusztig varieties  $X(\underline{w_1},\ldots, \underline{w_r})$  associated to a tuple of Weyl group elements, cf. Definition \ref{dfncompactifieddlvar}, \cite[Definition 2.3.2]{DMR}. 
This notion makes the construction of compactifications of Deligne--Lusztig varieties more straightforward and uniform. 
For example, setting $r=1$ yields the Zariski closure $X(\underline{w_1})=\overline{X(w_1)}$ of $X(w_1)$ in $G/B$. Setting $w_1,\ldots , w_r\in S$ such that $w_1\cdots w_r\in W$ is a reduced expression yields the smooth compactification  $X(\underline{w_1},\ldots, \underline{w_r})=\overline{X}(w_1\cdots w_r)$ defined by Deligne and Lusztig.

In this generality, we give an irreducibility criterion for  compactified Deligne--Lusztig varieties  $X(\underline{w_1},\ldots, \underline{w_r})$, generalising a result of Bonnaf\'{e} and Rouquier \cite{BonnafeRouquierOntheirred}. 
We also show that they have $F_p$-rational singularities (\Cref{Fpregularity}). %and thus are normal, Cohen--Macaulay, and pseudo-rational. 
%Furthermore, we provide a list of  compactifications $X(\underline{w_1},\ldots, \underline{w_r})$  of Deligne--Lusztig varieties which  are smooth. 

 The structure of the paper is as follows. 
 We fix our notations for Deligne--Lusztig varieties and compactified Deligne--Lusztig varieties and describe their geometric properties  in \Cref{secconnectednessandsmoothnesscriterions}. In  \Cref{secbiratandbraid}, we show that $X(\underline{w_1},\ldots, \underline{w_r})$ is smooth when each $w_i, i=1,\ldots, r,$ comes from a braid relation in the Weyl group $W$, and we establish relations between cohomology groups for the structure sheaf and the canonical sheaf on the associated compactified Deligne--Lusztig varieties.

\Cref{secP1fibandtwistedconj} contains the constructions of morphisms which are $\mathbb{P}^1$-fibrations between certain compactified Deligne--Lusztig varieties. They are associated with   $F$-cyclic conjugations in the Weyl group. In Proposition \ref{propcohwrtP1fibration}, we relate the cohomology of the structure sheaf  (resp. canonical sheaf) between these compactified Deligne--Lusztig varieties.
Finally, we prove \Cref{mainthm} in \Cref{seccohofcompactified}.

\section{Compactified Deligne--Lusztig varieties}\label{secconnectednessandsmoothnesscriterions}
We begin with fixing notations related to Deligne--Lusztig varieties associated to a tuple of Weyl group elements. This construction has been treated in \cite[\textsection 9]{DL} and \cite[\textsection 2]{DMR}. We review  the definitions and properties that are relevant for our context. 
Moreover, by adapting classical proofs, we give criterions for compactified Deligne--Lusztig varieties to be irreducible and show that they are $F_p$-rational. 

\subsection{Definitions and basic properties}\label{2.1definition}

Let $G$, $F:G\to G$, $T\subset B$, $S\subset W$ be as in \Cref{secbackground}. 
As  above, we take a suitable pair of $F$-stable $T\subset B$ such that 
$F(S)=S$ \cite[\textsection 4.2]{DigneMichelBook}.
In the literature, the endomorphism $F$ is referred to as a Frobenius root or a Steinberg endomorphism.
Denote by $\ell:W\to \mathbb{N}\cup \{0\}$ the Bruhat length and $\leq$ the Bruhat order on $W$.

Recall that by the Bruhat decomposition, the diagonal $G$-action on the product $G/B\times G/B$ yields a bijection between the Weyl group $W$ and $G\backslash\left(G/B\times G/B\right)$. The orbit corresponding to $w\in W$ is denoted by $\Orb{w}$. It can be regarded as a fibre bundle over the flag variety $G/B$ with fibres isomorphic to the Schubert cell $BwB/B$. 

Similarly,  for a given $r$-tuple $(w_1,\ldots, w_r) \in W^r$, one can construct an iterated fibre bundle over $G/B$ by taking fibre products of the orbits $O(w_i)$, $i=1,\ldots, r$, which we denote by
\begin{equation*}
		O(w_1,\ldots,w_r):=O(w_1)\times_{G/B}\cdots \times_{G/B}O(w_r).
	\end{equation*}

Moreover, we denote the fibre product of the Zariski closures $\overline{O(w_i)}\subseteq G/B\times G/B$ of the orbits by
\begin{equation*}
		O(\underline{w_1},\ldots, \underline{w_r}):=\overline{O(w_1)}\times_{G/B}\cdots\times_{G/B}\overline{O(w_r)}.
	\end{equation*}
	 	 We refer to them as the \emph{compactified orbit  associated to $(w_1,\ldots, w_r)$} in the rest of the article. 
The algebraic group 
$G$ acts diagonally on $O(w_1,\ldots,w_r)$ and $O(\underline{w_1},\ldots, \underline{w_r})$. 
	
	Since $O(\underline{w_1},\ldots, \underline{w_r})$ is an iterated fibre bundle over $G/B$ with fibres isomorphic to Schubert varieties $\overline{Bw_iB}/B$, cf. \cite[F.23]{Jantzen}, it is an integral, separated, normal, Cohen--Macaulay $k$-scheme locally of finite type. A proof of normality can be found in \cite[Lemma 2.2.13]{DMR}.  
	The Cohen--Macaulay property behaves well under base change \cite[Lemma 045T]{stacks-project}, so $O(\underline{w_1},\ldots, \underline{w_r})$ is Cohen--Macaulay follows from the classical result that Schubert varieties are Cohen--Macaulay (see the introduction of \cite{Hashimoto}).

Denote by $F: G/B\to G/B$ the map induced by the endomorphism $F:G\to G$ and $\Gamma_F\subseteq G/B\times G/B$ its graph. Since $F^d$ is a Frobenius endomorphism, we know that $F$ is surjective, finite (\cite[Proposition 8.1.2]{DigneMichelBook}), separated (\cite[Proposition 9.13]{GoerWedhAlgGeom1}), and purely inseparable (\cite[(3.7.6)]{EGA1}). Moreover, the induced map $F:W\to W$ is a group automorphism  \cite[Lemma 4.2.21]{DigneMichelBook}. 

\begin{definition}\label{dfncompactifieddlvar}
Let $w_1,\ldots, w_r\in W$. We define the \emph{Deligne--Lusztig variety associated to $(w_1,\ldots, w_r)$} as
the fibre product
\begin{equation*}
	X(w_1,\ldots ,w_r):= O(w_1,\ldots ,w_r)\times_{G/B\times G/B}\Gamma_F
\end{equation*}
with respect to the projection map $\proj_{0,r}:O(w_1,\ldots ,w_r)\to G/B\times G/B$ and the inclusion map $\Gamma_F\hookrightarrow G/B\times G/B$. 
Similarly, we define the
	\emph{compactified Deligne--Lusztig variety associated to  $(w_1,\ldots, w_r)$} as
	\begin{equation*}
		X(\underline{w_1},\ldots, \underline{w_r}):=O(\underline{w_1},\ldots, \underline{w_r})\times_{G/B\times G/B}\Gamma_F.
	\end{equation*}
	\end{definition}

Deligne--Lusztig varieties $X(w_1,\ldots,w_r)$ and compactified Deligne--Lusztig varieties $X(\underline{w_1},\ldots, \underline{w_r})$ are separated $k$-schemes locally of finite type. 
The diagonal $G^F$-action on $O(w_1,\ldots,w_r)$ (resp. $O(\underline{w_1},\ldots, \underline{w_r})$) induces a diagonal $G^F$-action on $X(w_1,\ldots,w_r)$ (resp. $X(\underline{w_1},\ldots, \underline{w_r})$). 

For any $w_1,\ldots, w_r\in W$, the $k$-scheme $X(w_1,\ldots,w_r)$ is smooth of dimension $\ell(w_1)+\cdots +\ell(w_r)$. If for all $ i=1,\ldots, r$, the associated Schubert variety $\overline{{B}w_i{B}}/{B}$ is smooth, then $X(\underline{w_1},\ldots, \underline{w_r})$ is smooth of dimension $\ell(w_1)+\cdots +\ell(w_r)$, cf.  \cite[Proposition 2.3.5, 2.3.6 (i)]{DMR}. In general, compactified Deligne--Lusztig varieties have $F_p$-rational singularities, as we will see in Proposition \ref{propstronglyFregular}.

Whenever $(w_1',\ldots, w_{r}')$ is a tuple satisfying $w_i'\leq w_i$ for all $i$, 
there is a $G^F$-equivariant embedding $X(w_1',\ldots, w_{r}')\hookrightarrow X(\underline{w_1},\ldots, \underline{w_r})$.
Moreover, we have a stratification 
\begin{equation*}
	X(\underline{w_1},\ldots, \underline{w_r})=\bigcup_{w_i'\leq w_i, \text{ for all }i}X(w_1',\ldots w_r').
\end{equation*}

When $s_i\in S$ for all $i=1,\ldots, r$, the tuple $(s_1,\ldots, s_r)$ can be regarded as an element in the free monoid generated by $S$, which we denote by $S^*$.  
For any $s\in S$, the Zariski closure  $\overline{BsB}=B\langle s \rangle B$
is the parabolic subgroup corresponding to the subgroup $\langle s \rangle \subset W$, cf. \cite[Corollary 2.2.10]{DMR}, so  the compactified Deligne--Lusztig variety $X(\underline{s_1},\ldots, \underline{s_r})$ is always smooth. To simplify the notations, given any $w=s_1\cdots s_m\in  S^* $ with $s_i\in S, i=1,\ldots m$, we denote $X(s_1,\ldots,s_m)$ by $X^{S^*}(w)$ and 
$X(\underline{s_1},\ldots, \underline{s_m})$ by $\overline{X}^{S^*}(w)$. The notations $O^{S^*}(w)$ and 
 $\overline{O}^{S^*}(w)$ for the orbits are defined analogously. 
 
 If $s_1\cdots s_r\in S^*$ is a reduced expression of an element in $w\in W$, then $\overline{X}^{S^*}(s_1\cdots s_r)$ recovers the smooth compactification $\overline{X}(w)$ as defined in \cite[\textsection 9.10]{DL}. 
Furthermore, denote by $\alpha: S^* \to W$  be the natural projection morphism of monoids. Let  $\preceq$ be the Bruhat order in $S^*$ and let $\ell^{S^*}:S^*\to \mathbb{N}\cup\{0\}$ be the Bruhat length.

\subsection{The Lang map}\label{dfnpartialfrob}
We  consider an alternative description of Deligne--Lusztig varieties via the Lang map. 
Using this description, we show that $G^F$ acts transitively on the set of irreducible components of  $X(w_1,\ldots,w_r)$  and $X(\underline{w_1},\ldots, \underline{w_r})$.   Denote by $G^r$ the $r$-fold product $G\times\cdots \times G$ over $\Spec   \ k$. This is again a $k$-group scheme.

We begin with fixing notations. 
Let  $F_1:G^r\to G^r$ be the endomorphism  given by
	\begin{equation*}
		(g_0,\ldots,g_{r-1})\longmapsto (g_1,\ldots, g_{r-1}, F(g_0)).
	\end{equation*} 
	Note that $F_1$ can be viewed as a Frobenius root for the algebraic group $G^r$. 
	
		The \emph{Lang map $\mathcal{L}_1$ associated to $F_1$} is the endomorphism $x\mapsto x^{-1}F_1(x)$ of $G^r$. In other words, we have
	\begin{equation*}
		\mathcal{L}_1:(g_0,\ldots,g_{r-1})\longmapsto (g_0^{-1}g_1,\ldots, g_{r-2}^{-1}g_{r-1}, g_{r-1}^{-1}F(g_0)).
	\end{equation*}

\begin{lemma}\label{lemequidimensionalityofcompactifieddlvar}
	 For any $w_1,\ldots,w_r\in W$, there are $G^F$-equivariant isomorphisms
	 	\begin{equation*}
		X(w_1,\ldots ,w_r)\overset{\sim}{\longrightarrow}\mathcal{L}_1^{-1}\left({B}w_1{B} \times_{\Spec \  k} \cdots \times_{\Spec   \  k} {B}w_r{B}\right)/({B})^r
	\end{equation*}
	and 
	 	\begin{equation*}
		X(\underline{w_1},\ldots, \underline{w_r})\overset{\sim}{\longrightarrow}\mathcal{L}_1^{-1}\left(\overline{{B}w_1{B}}\times_{\Spec  \  k} \cdots \times_{\Spec  \  k} \overline{{B}w_r{B}}\right)/({B})^r.
	\end{equation*}
	Moreover, $G^F$ acts transitively on the irreducible components of  $X(w_1,\ldots ,w_r)$ (resp. $X(\underline{w_1},\ldots, \underline{w_r})$). In particular, $X(w_1,\ldots ,w_r)$ and $X(\underline{w_1},\ldots, \underline{w_r})$ are equidimensional.
\end{lemma}
\begin{proof}
We record the proof for $X(\underline{w_1},\ldots, \underline{w_r})$ as the proof for $X(w_1,\ldots ,w_r)$ is the same. In this proof,
we denote $\mathcal{B}:=\overline{{B}w_1{B}}\times_{\Spec \  k} \cdots \times_{\Spec   \  k} \overline{{B}w_r{B}}$.

By Lang's theorem as formulated in \cite[Theorem 7.1]{CabanEngue}, we deduce that $\mathcal{L}_1:G^r\to G^r$ is surjective and is a Galois cover with Galois group $G^F$. 
The subscheme $\mathcal{L}_1^{-1}\left(\mathcal{B}\right)$ of $G^r$ is stable under the free $({B})^r$-action from the right, so we may take the quotient of $\mathcal{L}_1^{-1}\left(\mathcal{B}\right)$ with respect to the induced free right $({B})^r$-action. It follows that $\mathcal{L}_1^{-1}\left(\mathcal{B}\right)/({B})^r$ is a closed subscheme of $\left(G/B\right)^r$. Observe that $X(\underline{w_1},\ldots, \underline{w_r})$  is also isomorphic to a closed subscheme of $\left(G/B\right)^r$ via  projection to the first $r$-entries. The isomorphisms  in this lemma then follows from verifying the definitions, cf. 
 \cite[\textsection 1.11]{DL}. 
 
Moreover, 	
\begin{equation*}
		G^F\backslash \mathcal{L}_1^{-1}(\mathcal{B})\overset{\sim}{\longrightarrow}\mathcal{B}
	\end{equation*}
	is irreducible. Since the quotient map $\pi: \mathcal{L}_1^{-1}\left(\mathcal{B}\right)\to \mathcal{L}_1^{-1}\left(\mathcal{B}\right)/({B})^r$ is equivariant under the left $G^F$-action and surjective, it induces a surjective morphism $G^F\backslash \mathcal{L}_1^{-1}(\mathcal{B}) \to G^F\backslash\mathcal{L}_1^{-1}\left(\mathcal{B}\right)/({B})^r$. Thus 
	 $G^F\backslash\mathcal{L}_1^{-1}\left(\mathcal{B}\right)/({B})^r$ is irreducible and 
	  $G^F$ acts transitively on the set of irreducible components of $\mathcal{L}_1^{-1}\left(\mathcal{B}\right)/({B})^r$. 	

	\end{proof}

\begin{remark}
	The $G^F$-action on $X(\underline{w_1},\ldots, \underline{w_r})$ is not free. In $G/B$,  the point ${B}$ already has nontrivial stabiliser $B^F=B(\mathbb{F}_q)$.  The point $(B,\ldots, B)$ in $X(\underline{w_1},\ldots, \underline{w_r})$ is also stabilised by  $B^F=B(\mathbb{F}_q)$.
\end{remark}

\subsection{Criterion for irreducibility}
There is a simple criterion for the irreducibility of classical Deligne--Lusztig varieties, proved by Lusztig, Bonnaf\'{e}--Rouquier \cite{BonnafeRouquierOntheirred}, G\"{o}rtz \cite{GoertzOntheconn}, and many others.  
In this section we show that $X(w_1,\ldots,w_r)$ (resp. $X(\underline{w_1},\ldots, \underline{w_r})$) is irreducible when the tuple $(w_1,\ldots, w_r)$ has full $F$-support by modifying the proof of \cite{BonnafeRouquierOntheirred}. 

To begin, we fix the notations for the set of $F$-orbits in $S$ and the $F$-support for elements in $W$ and $S^*$. 

\begin{definition}\label{ForbitandFsupport}
 The \emph{$F$-orbit}  of $w\in W$ is the set $\OrbF{w}:=\{F^n(w)\vert  n\in \mathbb{Z}_{\geq 0}\}$. We define the set
\begin{equation*}
	S_F:=\{\OrbF{s} \vert  s\in S\}
\end{equation*}
	of $F$-orbits in $S$. 	
	For $w\in W$,  the \emph{$F$-support} of $w$ is the set
	\begin{equation*}
		{\supp}_F(w):=\{\tilde{s}\in S_F  \vert   \text{there exists }t\in \tilde{s}\text{ such that }t\leq w\}.
	\end{equation*}
	An element $w\in W$ is called \emph{of full $F$-support} if ${\supp}_F(w)=S_F$.
	
	 For a tuple $(w_1,\ldots w_r)\in W^r$, we denote
 \begin{equation*}
 	{\supp}_F(w_1,\ldots, w_r):={\supp}_F(w_1)\cup\cdots \cup {\supp}_F(w_r).
 \end{equation*}
 Similarly, the $F$-support for any $v\in S^*$ with $v=s_1\cdots s_m$ is defined as 
 	\begin{equation*}
		{\supp}_F(v):={\supp}_F(s_1,\ldots, s_m).
	\end{equation*}
\end{definition}
\begin{remark}\label{conventionforcoxeter}
	We say that $w\in W$ is a \emph{Coxeter element} when $w=s_1\cdots s_r$ is of full $F$-support and each $s_i\in S$ is from a distinct $F$-orbit.
\end{remark}

\begin{proposition}\label{propconnectednessandnumberofirreduciblecomponents}
Let $w_1,\ldots, w_r\in W$ and denote $I={\supp}_F(w_1,\ldots, w_r)$. 
The Deligne--Lusztig varieties $X(w_1,\ldots,w_r)$ and $X(\underline{w_1},\ldots, \underline{w_r})$ 
are irreducible if and only if $I=S_F$.  In general, they  have $\vert G^F/P_I^F\vert $ irreducible  components, where $W_I$ is the $F$-stable parabolic subgroup of $W$ generated by $\{s\in W \vert  \OrbF{s}\in I\}$ and 
$P_I={B}W_I{B}$.
\end{proposition}
\begin{proof} 
For certain reduced expressions of $w_1,\ldots, w_r$, take the product $w_1\cdots w_r$ in $S^*$ and denote it by $v$.
This product depends on the choices of the reduced expressions of $w_1,\ldots, w_r$, but  for each such $v$, there is an isomorphism 
 \begin{equation*}
 X^{S^*}(v)	\overset{\sim}{\longrightarrow}X(w_1,\ldots, w_r)
 \end{equation*}
 via projection, cf. \cite[\textsection 2.3.1]{DMR}. 
It also follows that $X(w_1,\ldots, w_r)$ and $X(\underline{w_1},\ldots, \underline{w_r})$ are  dense open subschemes of $\XF{v}$. 
 Thus it suffices to prove that 
 $\XF{v}$ is irreducible if and only if $v$ has full $F$-support. %GW 1.17
 
The proof showing that $\XF{v}$ is irreducible is analogous to the proof of \cite{BonnafeRouquierOntheirred} for classical Deligne--Lusztig varieties. 
 Suppose  $I=S_F$ and thus ${\supp}_F(v)=S_F$. Let $Z$ be the connected component of $\XF{v}$ which contains  the point $({B},\ldots, {B})$. Denote the stabiliser of $Z$ in $G^F$ by ${\rm Stab}(Z)$. Fix $v'$ to be of  minimal length such that $v'\preceq v$ and ${\supp}_F(v')=S_F$. In particular, $v'$ is a reduced expression of
$\alpha(v')\in W$, which is a Coxeter element. This implies that $\XF{v'}$ is irreducible \cite[Proposition 4.8]{LusztigCoxeterOrbits}.
 
 On the other hand, there is a $G^F$-equivariant embedding $\XF{v'}\to \XF{v}$. Let $Z'$ be the connected component of $\XF{v'}$ containing $({B},\ldots, {B})$. 
 Then ${\rm Stab}(Z')\subseteq {\rm Stab}(Z)$. However,  $\XF{v'}$ is irreducible, so ${\rm Stab}(Z')=G^F$. Thus $\XF{v}$ is irreducible. In this case,  its dense open subschemes $X(w_1,\ldots, w_r)$ and $X(\underline{w_1},\ldots, \underline{w_r})$ are  also irreducible \cite[Proposition 3.24]{GoerWedhAlgGeom1}. 
 
   Suppose  ${\supp}_F(w_1,\ldots, w_r)=I$ and so  ${\supp}_F(v)=I$. By \cite[Proposition 2.3.8]{DMR}, we know that there is a $G^F$-equivariant isomorphism
   \begin{equation*}
   	\XF{v}\overset{\sim}{\longrightarrow}G^F/U_I^F\times^{L_I^F}\overline{X}^{S^*}_{L_I}(v),
   \end{equation*}
   where $U_I$ is the unipotent radical of $P_I$ and $L_I$ is the Levi subgroup such that $P_I=U_I\rtimes L_I$. Recall that $L_I$ is reductive, and $\overline{X}^{S^*}_{L_I}(v)$ is the compactified Deligne--Lusztig variety for $L_I$ corresponding to $v$. Consider the $G^F$-equivariant quotient morphism
   \begin{equation*}
   	G^F/U_I^F\times^{L_I^F}\overline{X}^{S^*}_{L_I}(v)\longrightarrow G^F/P_I^F,
   \end{equation*}
   whose fibres are isomorphic to $\overline{X}^{S^*}_{L_I}(v)$. In this case, we know that  $v$ has full $F$-support with respect to the Weyl group of $L_I$, so $\XF{v}$ has  $\vert G^F/P_I^F\vert $ irreducible components.
 	\end{proof}

\subsection{$F_p$-regularity}\label{Fpregularity}

We denote by $F_p$ the (absolute) $p$-Frobenius to distinguish our notations from the notations coming from the theory of $F_p$-singularities in commutative algebra and algebraic geometry. 

In \cite[Proposition A.6]{Wang22}, we showed that some compactified Deligne--Lusztig varieties are strongly $F_p$-regular. This property generalises to arbitrary compactified Deligne--Lusztig varieties. As in loc. cit., we use a similar strategy as in the proof of \cite[Lemma 4.3]{Lusztigunipotentch}, which 
uses faithfully flat descent of the smoothness property along the Lang map.

\begin{proposition}\label{propstronglyFregular}
For $w_1,\ldots, w_r\in W$, the compactified Deligne--Lusztig variety $X(\underline{w_1},\ldots, \underline{w_r})$ is strongly $F_p$-regular. In particular, $X(\underline{w_1},\ldots, \underline{w_r})$ is $F_p$-rational and thus normal and Cohen-Macaulay. 

If $\overline{Bw_iB}$ is smooth for all $i$, then $X(\underline{w_1},\ldots, \underline{w_r})$ is smooth. 
\end{proposition}
\begin{proof}
When $\overline{Bw_iB}$ is smooth for all $i$, this is proved by faithfully flat descent of smoothness in \cite[Lemma 4.3]{Lusztigunipotentch} and \cite[Proposition 2.3.5]{DMR}, cf.  \cite[Proposition 14.59]{GoerWedhAlgGeom1}. Analogously, we use faithfully flat descent of strongly $F_p$-regularity to show that any compactified Deligne--Lusztig variety is strongly $F_p$-regular.

	For each $i=1,\ldots, r$, we know that $\overline{{B}w_i{B}}$ is strongly $F_p$-regular \cite[Cor. 4.1, p. 958]{BrTh} and thus Cohen-Macaulay \cite[Cor. 4.2, p. 959]{BrTh}. The product
	\begin{equation*}
		\overline{{B}w_1{B}}\times_{\Spec \    k} \cdots \times_{\Spec \  k} \overline{{B}w_r{B}}
	\end{equation*}
	is strongly $F_p$-regular by faithfully flat descent along the projection map ${\rm pr}_i: \overline{{B}w_1{B}}\times_{\Spec \    k} \cdots \times_{\Spec \  k} \overline{{B}w_r{B}}\to \overline{{B}w_i{B}}$ \cite[Theorem 5.5.b]{HoHu}.
	 Recall that the Lang map $\mathcal{L}_1$ is \'{e}tale, so its closed fibres are smooth and thus strongly $F_p$-regular, cf. \cite[Ch 3, Remark 1.6, Theorem 4.4(a)]{KSmith}. We may now apply \cite[Theorem 3.6]{Aber} and conclude that 
	\begin{equation*}
		\mathcal{L}_1^{-1}\left(\overline{{B}w_1{B}}\times_{\Spec \  k} \cdots \times_{\Spec   \  k} \overline{{B}w_r{B}}\right)
	\end{equation*}
	is strongly $F_p$-regular. Finally, by Lemma \ref{lemequidimensionalityofcompactifieddlvar} and the faithfully flat descent of strong $F_p$-regularity \cite[Theorem 5.5]{HoHu}, $X(\underline{w_1},\ldots, \underline{w_r})$ is strongly $F_p$-regular. 
		
	Strongly $F_p$-regular rings are $F_p$-rational, so  $X(\underline{w_1},\ldots, \underline{w_r})$  is normal and Cohen-Macaulay by \cite[Theorem 4.2, 6.27]{HoHu}.
			\end{proof}

\begin{remark}
	For $w\in W$, the Zariski closure $\overline{X(w)}$ in $G/B$ is not always smooth, but for 
	any $v\in S^*$ that gives a reduced expression of $w$, the projection map $\XF{v}\to \overline{X(w)}$ is a proper birational morphism \cite[Lemma 9.11]{DL}.

\end{remark}

\section{Braid relations in the Weyl group and birational morphisms}\label{secbiratandbraid}
\subsection{Smoothness via braid relations}
We  consider some specific elements in the Weyl group for which the compactified Deligne--Lusztig variety is smooth. To begin, we recall the relations in the Weyl group of finite reductive groups, cf. \cite{GeckPfeiffer}. 

All the relations in the group presentation of the Weyl group as a finite Coxeter group $(W,S)$ can be written as $(st)^{m_{st}}=1,  s,t\in S,$
where $m_{st}$ is a positive integer such that $m_{ss}=1$ and $m_{st}\geq 2 $ if $s\neq t$. 

When  $(st)^{m}=1$, with $ \ s,t\in S$ and  $s\neq t$ is a relation in $W$, we define 
\begin{equation*}
		P(s,t):=st\cdots \in S^*,
	\end{equation*}
	such that $\ellm{P(s,t)}=m$. When there is no ambiguity, we also use $P(s,t)$ to denote its canonical projection $\alpha(P(s,t))\in W$. Moreover, the equation $(st)^{m}=1$ can be rewritten as 
\begin{equation}\label{grouppresrel}
	P(s,t)=P(t,s)
\end{equation}
with $\ellm{P(s,t)}=\ellm{P(t,s)}=m$. We call (\ref{grouppresrel}) 
a \emph{braid relation} for $W$ following the convention of \cite[\textsection 4]{GeckPfeiffer}.

\begin{lemma}\label{grouppresrelimplysmoothness}
	Let $w_1,\ldots, w_r\in W$. Suppose for each $j\in \{1,\ldots, r\}$, there is a braid relation (\ref{grouppresrel}) of $W$  such that $w_j=P(s,t)=P(t,s)$.
	 Then the compactified Deligne--Lusztig variety $X(\underline{w_1},\ldots, \underline{w_r})$ is smooth. 
\end{lemma}
\begin{proof}
By Proposition \ref{propstronglyFregular}, it suffices to show that for all $j=1,\ldots, r$, the associated scheme $\overline{{B}w_j{B}}$ is smooth. 

Consider $w_j=P(s,t)=P(t,s)$ given by $(st)^m=1$. Denote $J=\{s,t\}\subseteq S$.
The the corresponding subgroup $W_J\subseteq W$  generated by $J$ has a group presentation
 \begin{equation*}
 W_J=\langle s, t \vert (st)^{m}=1, s^2=1, t^2=1 \rangle.	
 \end{equation*}
 As a consequence, $P(s,t)$ is a reduced expression of the longest  element of $W_J$. 
 
 Let $P_J=BW_JB\subseteq G$ be the parabolic subgroup corresponding to $W_J$.
 Since $w_j$ is the longest element in $W_J$, we know that $\overline{{B}w_j{B}} = P_J$ and is smooth cf. \cite[Proposition 3.2.8]{DigneMichelBook}.
 \end{proof}
\subsection{Higher direct images of the structure sheaf for birational morphisms}

Recall that the braid monoid $B^+=B^+(W,S)$ associated to $W$ is a monoid generated by $S$ with respect to the  braid relations (\ref{grouppresrel}) for $W$, cf. \cite[\textsection 4]{GeckPfeiffer}. There is a natural surjective morphism of monoids
$\beta:S^*\twoheadrightarrow B^+$. Note that the projection map $\alpha: S^* \to W$ from \textsection \ref{secconnectednessandsmoothnesscriterions} factors through $\beta$.

\begin{proposition}\label{propcohwrtbraidrelations}
Let $w,v\in S^*$ such that $\beta(w)=\beta(v)$. Then for all $i\geq 0$,  there are $G^F$-equivariant isomorphisms
	\begin{equation*}
 	H^i\left(\XF{w},\Osh_{\XF{w}}\right)\overset{\sim}{\longrightarrow}H^i\left(\XF{v},\Osh_{\XF{v}} \right)
 \end{equation*}
 and 
 \begin{equation*}
 	H^i\left(\XF{w},\omega_{\XF{w}}\right)\overset{\sim}{\longrightarrow}H^i\left(\XF{v},\omega_{\XF{v}}\right).
 \end{equation*}	
\end{proposition}
\begin{proof}
	We can rewrite $w=w_1\cdots w_r$ in $S^*$ such that
	$\ellm{w_j}=\ell(\alpha(w_j))$ and $\overline{{B}\alpha(w_j){B}}$  is smooth for all $j=1,\ldots, r$. Suppose for an $l\in \{1,\ldots, r\}$,  we have 
	$w_l=P(s,t)$ given by 
	 a braid relation $P(s,t)=P(t,s)$ of $W$. Define the element $w':=w_1\cdots P(t,s)\cdots w_r$ in $S^*$ by replacing $w_l$ by $P(t,s)$. We show below that $\XF{w}$, $\XF{w'}$, and $X(\underline{\alpha(w_1)},\ldots,\underline{\alpha(w_r}))$ have the same cohomology for the structure sheaf and canonical sheaf. Since $\beta(w)=\beta(v)$, the statement of the proposition follows from applying this procedure finitiely many times.

By assumption, we have $\alpha(w_l)=P(s,t)=P(t,s)$ in $W$, so replacing $w_l$ by $P(t,s)$ does not affect the construction of the following compactified Deligne--Lusztig variety
	\begin{equation*}
		X(\underline{\alpha(w_1)},\ldots,\underline{\alpha(w_r}))=X(\underline{\alpha(w_1)},\ldots,\underline{P(t,s)},\ldots,\underline{\alpha(w_r})).
	\end{equation*} 
	For $X_n=\XF{w}$, $\XF{w'}$,  $X(\underline{\alpha(w_1)},\ldots,\underline{\alpha(w_r}))$, it follows from the assumptions   and Lemma \ref{grouppresrelimplysmoothness} that $X_n$ is  smooth and projective. 
Note that we have $I:=\supp_F(w)=\supp_F(w')=\supp_F(\alpha(w_1),\ldots,\alpha(w_r))$, so 
 $X_n$ has $|G^F/P_I^F|$ irreducible components by Proposition \ref{propconnectednessandnumberofirreduciblecomponents},. 	

	There are natural projection maps ${\rm pr}_{0,\ell(w_j)}:\overline{O}^{S^*}(w_j)\to O(\underline{w_j})$  for all $j=1,\ldots r$ and ${\rm pr}_{0,\ell(w_l)}:\overline{O}^{S^*}(P(t,s))\to O(\underline{\alpha(w_l)})$.
	Similarly, the projection to the $0$-th, $\ell(w_1)$-th, $\ell(w_1)+\ell(w_2)$-th,\ .\ .\ ., and $\ell(w_1)+\cdots+\ell(w_{r})$-th coordinate yields the vertical maps in the cartesian diagrams
		\begin{equation*}
\begin{tikzcd}
	\XF{w} \arrow[r]\arrow[d,"f"'] &\overline{O}^{S^*}(w)\arrow[d]\\
	X(\underline{\alpha(w_1)},\ldots,\underline{\alpha(w_r})) \arrow[r, hook] & O(\underline{\alpha(w_1)},\ldots,\underline{\alpha(w_r})),
\end{tikzcd}
\end{equation*}
and
	\begin{equation*}
\begin{tikzcd}
	\XF{w'} \arrow[r]\arrow[d,"f'"'] &\overline{O}^{S^*}(w')\arrow[d]\\
	X(\underline{\alpha(w_1)},\ldots,\underline{\alpha(w_r})) \arrow[r, hook] & O(\underline{\alpha(w_1)},\ldots,\underline{\alpha(w_r})).
\end{tikzcd}
\end{equation*}
	The maps $f$ and $f'$ are $G^F$-equivariant proper birational morphisms. 
	
	It follows from the proof of Zariski's main theorem (\cite[Corollary 11.4]{Hartshorne}) that $\phi_*\Osh_{X_1}\cong \Osh_{X_2}$,  for  $\phi=f$ (resp. $\phi=f'$), $X_1=\XF{w}$ (resp. $X_1=\XF{w'}$) and $X_2=X(\underline{\alpha(w_1)},\ldots,\underline{\alpha(w_r}))$. 
	Moreover, applying \cite[Corollary 3.2.10]{CR11} 	yields
	$R\phi_*\Osh_{X_1}\cong \Osh_{X_2}$. Hence the Leray spectral sequence associated to $\phi$ degenerates. 
	
	Since $\phi$ is $G^F$-equivariant, the isomorphism $\phi_*\Osh_{X_1}\cong \Osh_{X_2}$ is $G^F$-equivariant, 
	thus functoriality of the associated Leray spectral sequence yields $G^F$-equivariant isomorphisms on cohomology groups
		\begin{equation*}
 	H^i\left(X_1,\Osh_{X_1}\right)\overset{\sim}{\longrightarrow}H^i\left(X_2,\Osh_{X_2} \right),
 \end{equation*}
	for  all $i\geq 0$.
	The $G^F$-equivariant isomorphisms for the cohomology of the canonical sheaf follow from an equivariant version of Serre duality \cite[Corollary 2.10]{Nironi}.
	\end{proof}

\section{$\mathbb{P}^1$-fibrations and $F$-conjugacy classes}\label{secP1fibandtwistedconj}
We construct $\mathbb{P}^1$-fibrations between compactified Deligne--Lusztig varieties and relate their cohomology for the structure sheaf and the canonical sheaf. 
Moreover, we use these $\mathbb{P}^1$-fibrations to show that pairs of compactified Deligne--Lusztig varieties of the form $\XF{sv}$ and $\XF{vF(s)}$ have isomorphic cohomology groups for the the structure sheaf and the canonical sheaf.

\subsection{$\mathbb{P}^1$-fibrations over Deligne--Lusztig varieties}
\begin{lemma}\label{lemswFstowFsP1}
	Let  $w_1,\ldots, w_r, w_1',\ldots, w_m'\in W$ and $s\in S$. The projection map  
	\begin{equation*}
		\phi: X(\underline{w_1},\ldots, \underline{w_{r}},\underline{s}, \underline{s}, \underline{w_{1}'},\ldots,\underline{w_{m}'})\longrightarrow X(\underline{w_1},\ldots, \underline{w_{r}},\underline{s}, \underline{w_{1}'},\ldots,\underline{w_{m}'})
	\end{equation*}
	given by $(B_0,\ldots, B_{r}, B_{r+1}, B_{r+2}, \ldots, B_{r+m+1}, FB_0)\mapsto (B_0,\ldots, B_{r}, B_{r+2}, \ldots, B_{r+m+1}, FB_0)$ is a $G^F$-equivariant Zariski $\mathbb{P}^1$-bundle. 
	
	The $G^F$-equivariant morphism of $k$-schemes
	\begin{equation*}
		\pi: X(\underline{s}, \underline{w_1},\ldots, \underline{w_r},\underline{F(s)})\longrightarrow X(\underline{w_1},\ldots, \underline{w_r},\underline{F(s)}),
	\end{equation*} 
	given by $(B_0,\ldots, B_{r+1}, FB_0)\mapsto (B_1,\ldots, B_{r+1}, FB_1)$, 
	 is a Zariski $\mathbb{P}^1$-bundle. 
	\end{lemma}
\begin{proof}
	For any $t\in S$, the projection map $ O(\underline{t},\underline{t})\to O(\underline{t})$ given by $(B_0, B_1,B_2)\mapsto (B_0, B_2)$ is a $\mathbb{P}^1$-bundle. Recall that it is a property of the Bruhat decomposition that if $(B_0,B_1)\in O(t)$ and $(B_1, B_2)\in O(t)$ then $(B_0,B_2)\in O(\underline{t})$, cf. \cite[\textsection 1.2]{DL}. 
	
	For $t=s$, we see that $\phi$ is the base change of $ O(\underline{t},\underline{t})\to O(\underline{t})$ along the projection map $X(\underline{w_1},\ldots, \underline{w_{r}},\underline{t}, \underline{w_{1}'},\ldots,\underline{w_{m}'})\to O(\underline{t})$ given by $(B_0',\ldots,FB_0')\mapsto (B_{r}', B_{r+1}')$. The first claim follows. 
		
	Consider $t=F(s)\in S$.  One can verify that $\pi$ is $G^F$-equivariant by its definition.  
Let $(B_0,B_1)\in O(\underline{s})$, so $(FB_0, FB_1)\in O(\underline{F(s)})$. Recall that via the projection to the first $(r+2)$-entries, $X(\underline{s},\underline{w_1},\ldots, \underline{w_r},\underline{F(s)})$ can be identified with a closed subscheme of $(G/B)^{r+2}$. Under this identification, we have an embedding
\begin{equation*}
\iota: X(\underline{s},\underline{w_1},\ldots, \underline{w_r},\underline{F(s)})\longrightarrow
	O(\underline{s})\times_{G/B} X(\underline{w_1},\ldots, \underline{w_r},\underline{F(s)}), 
\end{equation*}
given by $(B_0,\ldots, B_{r+1})\mapsto ((B_0,B_1),(B_1,\ldots, B_{r+1}))$.

Since $F(s)\in S$, for any $(B_0,B_1)\in O(\underline{s})$, if there exists a $B'\in G/B$ with $(B',FB_1)\in O(\underline{F(s)}) $, then we automatically have $(B',FB_0)\in O(\underline{F(s)})$. Thus $\iota$ is surjective with inverse given by 
$((B_0',B_1'),(B_1',\ldots, B_{r+1}'))\mapsto (B_0',\ldots, B_{r+1}')$. 
Therefore $X(\underline{s}, \underline{w_1},\ldots, \underline{w_r},\underline{F(s)})$ is isomorphic to $O(\underline{s})\times_{G/B} X(\underline{w_1},\ldots, \underline{w_r},\underline{F(s)})$. It follows that $\pi$ 
is the composition of the isomorphism $\iota$ with
 the projection 
\begin{equation*}
	O(\underline{s})\times_{G/B} X(\underline{w_1},\ldots, \underline{w_r},\underline{F(s)})\longrightarrow X(\underline{w_1},\ldots, \underline{w_r},\underline{F(s)}),
\end{equation*}
which is a Zariski $\mathbb{P}^1$-bundle. 
\end{proof}
\begin{remark}
	The reason for
 the fibre product of $X(\underline{w_1},\ldots, \underline{w_r},\underline{F(s)})$ with  $O(\underline{s})$  to yield a Deligne--Lusztig variety is that 
the orbit of pairs $(B_0,B_1)$ determines the orbit of  $(FB_0,FB_1)$. This allows us to verify that the tuples $(B_0,\ldots , B_{r+1}, FB_0)$  live in the fibre product  of the graph of $F$ and  a product of orbits. In general, taking a fibre product of a Deligne-Lusztig variety with  $O(\underline{s})$ does not necessarily yield another Deligne--Lusztig variety.
	\end{remark}

\begin{lemma}\label{lemtwistedP1fibration}
	Let  $w_1,\ldots, w_r\in W$ and $s\in S$ such that  $\overline{Bw_nB}$ is smooth for $n=1,\ldots, r$.  The  morphisms
	\begin{equation*}
		\tau_2: X(\underline{F^{-1}(s)}, \underline{F^{-1}(w_1)},\ldots, \underline{F^{-1}(w_r)},\underline{s}) \longrightarrow X(\underline{s}, \underline{w_1},\ldots, \underline{w_r}),
	\end{equation*}
	given by $(B_0,B_1,\ldots,B_{r+1},FB_0)\longmapsto (B_{r+1},FB_1,\ldots , FB_{r+1})$,
	and
	\begin{equation*}
	\tau_1:	X(\underline{F^{-1}(s)}, \underline{F^{-1}(w_1)},\ldots, \underline{F^{-1}(w_r)},\underline{s})  \longrightarrow X(\underline{w_1},\ldots, \underline{w_r},\underline{F(s)}),
	\end{equation*}
	given by $(B_0,B_1,\ldots,B_{r+1},FB_0)\longmapsto (FB_1,FB_2\ldots , F^2B_{1})$, are $G^F$-equivariant   fppf $\mathbb{P}^1$-bundles.
\end{lemma}
\begin{proof}
Given $(B_0,B_1)\in O(\underline{s})$, we know that $(FB_0,FB_1)\in O(F(\underline{s}))$. By construction we have $(B_{r+1},FB_0)\in O(F(\underline{s}))$, so $(B_{r+1},FB_1)\in O(F(\underline{s}))$ and $(FB_{r+1},F^2B_1)\in O(F(\underline{s}))$ \cite[\textsection 1.2]{DL}. Thus $\tau_1$ and $\tau_2$ are well-defined. They are  $G^F$-equivariant by construction. 
 
  Let $Z$ be  the fibre product
\begin{equation}\label{fibreproductforz}
\begin{tikzcd}
Z \arrow[r, "r_1"]\arrow[d,"d_1"'] & X(\underline{F^{-1}(s)}, \underline{F^{-1}(w_1)},\ldots, \underline{F^{-1}(w_r)},\underline{s}) \arrow[d, "\tau_2"]
\\
X(\underline{F^{-1}(s)}, \underline{F^{-1}(w_1)},\ldots, \underline{F^{-1}(w_r)})  \arrow[r, "F"] & X(\underline{s}, \underline{w_1},\ldots, \underline{w_r}).
\end{tikzcd}
\end{equation}
Geometric points in $Z$ are denoted as tuples 
\begin{equation*}
	((B_0,B_1,\ldots,B_{r+1},FB_0), (B_0',B_1',\ldots,B_{r}',FB_0')),
\end{equation*}
  where $(B_0,B_1,\ldots,B_{r+1},FB_0)\in  X(\underline{F^{-1}(s)}, \underline{F^{-1}(w_1)},\ldots, \underline{F^{-1}(w_r)},\underline{s})$  and\\
   $(B_0',B_1',\ldots,B_{r}',FB_0')\in X(\underline{F^{-1}(s)}, \underline{F^{-1}(w_1)},\ldots, \underline{F^{-1}(w_r)})$ such that 
$FB_0'=B_{r+1}$, $FB_1'=FB_1, \ldots, FB_r'=FB_{r} , F^2B_0'=FB_{r+1}$. 
Here the only condition on $B_0$ is that $(B_0,B_0')\in O(\underline{F^{-1}(s)})$.

The morphism $d_1$ also fits into another commutative diagram
\begin{equation}\label{fiberprodP1fibZ}
\begin{tikzcd}
Z \arrow[r, "r_2"]\arrow[d,"d_1"'] 
& O(\underline{F^{-1}(s)})\arrow[d, "{\rm pr}_1"] \\
 X(\underline{F^{-1}(s)}, \underline{F^{-1}(w_1)},\ldots, \underline{F^{-1}(w_r)})  \arrow[r, "{\rm pr}_0"]
&  G/B ,
\end{tikzcd}
\end{equation}
where ${\rm pr}_1:O(\underline{s}) \to G/B$ is given by $(B_0,B_1)\mapsto B_1$, and 
$r_2$ is  given by 
	\begin{equation*}
	((B_0,B_1,\ldots,B_{r+1},FB_0), (B_0',B_1',\ldots,B_{r}',FB_0')) \longmapsto (B_0,B_0').
\end{equation*}
The diagram (\ref{fiberprodP1fibZ}) yields a morphism
\begin{equation*}
\iota:	Z \longrightarrow O(\underline{F^{-1}(s)})\times_{G/B}X(\underline{F^{-1}(s)}, \underline{F^{-1}(w_1)},\ldots, \underline{F^{-1}(w_r)}).
\end{equation*}
Given $(B_0',B_1',\ldots,B_{r}',FB_0')\in X(\underline{F^{-1}(s)}, \underline{F^{-1}(w_1)},\ldots, \underline{F^{-1}(w_r)})$ and arbitrary $B_0\in G/B$ with $(B_0,B_0')\in O(\underline{F^{-1}(s)})$, we have $(B_0,B_1')\in O(\underline{F^{-1}(s)})$ and $(FB_0,FB_0')\in O(\underline{s})$. This gives
\begin{equation*}
	((B_0,B_1',\ldots,B_{r}',FB_0',FB_0),(B_0',B_1',\ldots,B_{r}',FB_0'))\in Z,
\end{equation*}
and that $\iota$ is surjective. By the Bruhat decomposition, for every $i=1,\ldots, r$, there exists $v_i\in W$ such that 
$(B_i, B_i')\in O(v_i)$. This implies $(FB_i, FB_i')\in O(F(v_i))$. We know $FB_i=FB_i'$ by assumption, so $F(v_i)=e$ is the neutral element in $W$ and  $v_i=e$ because $F:W\to W$ is an automorphism,  cf.  \cite[\textsection 1.2]{DL}. It follows that $\iota$ is also injective and an isomorphism.  

Hence  $d_1$ is a Zariski $\mathbb{P}^1$-bundle.
Since $F:X(\underline{F^{-1}(s)}, \underline{F^{-1}(w_1)},\ldots, \underline{F^{-1}(w_r)})\to X(\underline{s}, \underline{w_1},\ldots, \underline{w_r})$ is faithfully flat by our assumption on smoothness \cite[Corollary 14.129]{GoerWedhAlgGeom1}, 
 we deduce that $\tau_2$ is faithfully flat by faithfully flat descent \cite[Corollary 14.12, Proposition 14.3(2)]{GoerWedhAlgGeom1}.  

For each  affine open (Zariski)  covering $\{U_{\alpha}\}_{\alpha}$ of $X(\underline{F^{-1}(s)}, \underline{F^{-1}(w_1)},\ldots, \underline{F^{-1}(w_r)})$ that trivialises $d_1$, we can construct
 an fppf covering $\{\left.F\right|_{U_\alpha}:U_{\alpha}\to X(\underline{s}, \underline{w_1},\ldots, \underline{w_r})\}_{\alpha}$ of $X(\underline{s}, \underline{w_1},\ldots, \underline{w_r})$ over which $\tau_2$ trivialises as a $\mathbb{P}^1$-bundle.

 After taking the fibre product of $\tau_1$ with  $F: X(\underline{F^{-1}(w_1)},\ldots, \underline{F^{-1}(w_r)},\underline{s}) \to X(\underline{w_1},\ldots, \underline{w_r},\underline{F(s)})$, the claim for $\tau_1$ can be verified by the analogous argument we used for $\tau_2$ above. 
\end{proof}

\begin{proposition}\label{propcohwrtP1fibration}
Let  $w_1,\ldots, w_r\in W$ and $s\in S$ such that  $\overline{Bw_nB}$ is smooth for $n=1,\ldots, r$. 
There are  $G^F$-equivariant isomorphisms for all $i\geq 0$,
	\begin{equation*}
		H^i\big(X_1,\mathcal{O}_{X_1}\big)\overset{\sim}{\longrightarrow} H^i\big(X_2,\mathcal{O}_{X_2}\big)
	\end{equation*}
	and 
		\begin{equation*}
			H^i\big(X_1,\omega_{X_1}\big)\overset{\sim}{\longrightarrow} H^i\big(X_2,\omega_{X_2}\big)
	\end{equation*}
		induced by  $f:X_1\to X_2$, where $f$ is one of the $\mathbb{P}^1$-bundles  $\phi, \pi, \tau_1$ or $\tau_2$ from Lemma \ref{lemswFstowFsP1} and \ref{lemtwistedP1fibration}. 
\end{proposition}

	\begin{proof}
	The proof is simpler for the Zariski $\mathbb{P}^1$-bundles $\phi$ and $\pi$, so we only record the proof for the fppf $\mathbb{P}^1$-bundles $\tau_1$ and $\tau_2$.
		We use the notations of Lemma \ref{lemtwistedP1fibration}. Since $d_1$ is a Zariski $\mathbb{P}^1$-bundle, it is proper and faithfully flat. Hence  $\tau_1$ and $\tau_2$ are proper and faithfully flat by faithfully flat descent \cite[Proposition 14.53]{GoerWedhAlgGeom1}. 
              
Set $f:=\tau_2$ (resp. $\tau_1$), $X_1:=X(\underline{F^{-1}(s)}, \underline{F^{-1}(w_1)},\ldots, \underline{F^{-1}(w_r)},\underline{s})$, $X_2:=X(\underline{s}, \underline{w_1},\ldots, \underline{w_r})$ (resp. $X( \underline{w_1},\ldots, \underline{w_r},\underline{F(s)})$), and $X_2':=X(\underline{F^{-1}(s)}, \underline{F^{-1}(w_1)},\ldots, \underline{F^{-1}(w_r)})$ (resp. $X(\underline{F^{-1}(w_1)},\ldots, \underline{F^{-1}(w_r)},\underline{s})$).
Consider the Leray spectral sequence for $\mathcal O_{X_1}$:
		\begin{equation*}
			E_2^{i,j}=H^i\leftp X_2,R^jf_*\Osh_{X_1}\rightp \implies H^{i+j}\leftp X_1,\Osh_{X_1}\rightp.
		\end{equation*}
		
		Let $x\in X_2$ and $y\in X_2'$ such that $F(y)=x$ as in diagram (\ref{fibreproductforz}). By Lemma \ref{lemtwistedP1fibration}, the base change of the fibre $x\times_{X_2}X_1$ to $y$ is isomorphic to $\mathbb{P}^1_{y}$. It follows from flat base change that $H^j(x\times_{X_2}X_1, \Osh_{x\times_{X_2}X_1})=0$ for all $j\geq 0$ \cite[Corollary 22.91]{GoerWedhAlgGeom2}. 
		We conclude by Grauert's theorem \cite[Proposition 23.142]{GoerWedhAlgGeom2} (cf. \cite[Proposition 3.35]{GoerWedhAlgGeom1}) that $R^jf_*\Osh_{X_1}=0$ for all $j>0$. Hence the Leray spectral sequence degenerates.
		
		We have an fppf covering $\mathfrak{U}=\{\gamma_{\alpha}:U_{\alpha}\to X_2\}_{\alpha}$ of $X_2$ in Lemma \ref{lemtwistedP1fibration}, such that the base change 
		 $f_{\alpha}:X_1\times_{X_2}U_{\alpha}\to U_{\alpha}$  of $f$ along $\gamma_{\alpha}$ is a trivial $\mathbb{P}^1$-bundle for all $\alpha$. 
		 Since $U_{\alpha}$ is affine and
		 $\gamma_{\alpha}$ is an affine morphism for all $\alpha$  \cite[Proposition 12.3(3)]{GoerWedhAlgGeom1},
		  any simplex $V$ in the \v{C}ech nerve of $\mathfrak{U}$ is affine and for any coherent sheaf $\mathcal{F}$ on $X_2$, we have $H^i(X_2, \mathcal{F})\overset{\sim}{\longrightarrow}\check{H}^i(\mathfrak{U},\mathcal{F})$ for all $i\geq 0$ \cite[Theorem II.5.4.1, Corollary]{Godement}. Note that $f_*\Osh_{X_1}$ is a coherent $\Osh_{X_2}$-module because $f$ is proper and cohomology of coherent sheaves is independent of  topology \cite[03DW]{stacks-project}. 
		  
		  For each simplex $V$ in the \v{C}ech nerve of $\mathfrak{U}$, we have $V\times_{X_2}X_1\overset{\sim}{\longrightarrow}V\times_{\Spec \ k}\mathbb{P}^1$, so $f_*\Osh_{X_1}(V)=\Osh_V(V)$ by
		flat base change \cite[Proposition 22.90]{GoerWedhAlgGeom2}. This identification is compatible with restriction of sections, so the \v{C}ech complexes on $\mathfrak{U}$ for $f_*\Osh_{X_1}$ and $\Osh_{X_2}$ are identical and $H^i(X_2, \Osh_{X_2})\overset{\sim}{\longrightarrow} H^i(X_2, f_*\Osh_{X_1})$ for all $i\geq 0$. 
		
	Finally, degeneration of the Leray spectral sequence  yields	
 $G^F$-equivariant 
isomorphisms
								\begin{equation*}
					H^i\leftp X_2,\Osh_{X_2}\rightp \overset{\sim}{\longrightarrow} H^{i}\leftp X_1,\Osh_{X_1}\rightp,
				\end{equation*}
				for all $i\geq 0$.

	The $G^F$-equivariant isomorphisms for the cohomology of $\omega_{X_1}$ and $\omega_{X_2}$ follow from Serre duality for Deligne--Mumford stacks \cite[Corollary 2.10]{Nironi}.
	 
	\end{proof}
	
	\begin{corollary}
		\label{corcohisomcyclicshiftingoperator}
Let $s\in S$. Let $w=sv$ and $w'=vF(s)$ be elements of $S^*$. There are  $G^F$-equivariant isomorphisms for all $i\geq 0$,
	\begin{equation*}
		H^i\Big(\XF{w}, \mathcal{O}_{\XF{w}}\Big)\overset{\sim}{\longrightarrow}H^i\Big(\XF{w'}, \mathcal{O}_{\XF{w'}}\Big)
	\end{equation*}
	and 
		\begin{equation*}
		H^i\Big(\XF{w}, \omega_{\XF{w}}\Big)\overset{\sim}{\longrightarrow}H^i\Big(\XF{w'}, \omega_{\XF{w'}}\Big).
	\end{equation*}	
		\end{corollary}
\begin{proof}
	This follows from Proposition \ref{propcohwrtP1fibration} by composing the corresponding isomorphisms of cohomology groups for $\tau_1$ and $\tau_2$. 
	\end{proof}

\subsection{$F$-conjugacy classes in $W$}\label{subsecFconjugacyclassesandFcyclicshifts}
 Two elements $w,w'\in W$ are $F$-conjugate if $w'=v^{-1}wF(v)$ for some $v\in W$. 
 The $F$-conjugacy classes in $W$ give a partition of $W$. An element $w\in W$ is of \emph{minimal length} if it has the smallest Bruhat length in its $F$-conjugacy class. 
 
 Minimal length elements were studied by Geck, Kim, and Pfeiffer \cite{GeckKimPfeifferMinlength}, as well as He and Nie \cite{HeNieminlength}.  In our context, the results of these two papers imply that 
 that for any $w\in W$, there exists a
minimal length  element $w'$  in the $F$-conjugacy class of $w$ such that $w\to_F w'$ (cf. Definition \ref{dfncyclicshift}).  In combination with Proposition \ref{propcohwrtP1fibration}, we obtain a framework for  an inductive approach for studying cohomology groups of Deligne--Lusztig varieties.

\begin{definition}\label{dfncyclicshift}
	We write $w\to_F w'$ if there exist $w_1,\ldots, w_r\in W$ and $s_1,\ldots, s_{r-1}\in S$ such that $w_1=w, w_r=w'$, $w_{i}=s_{i-1}w_{i-1}F(s_{i-1})$, and $\ell(w_i)\leq \ell(w_{i-1})$ for all $i=2,\ldots, r$.
\end{definition}

Furthermore, for any pairs of minimal length elements $w$ and $w'$ in an \emph{elliptic $F$-conjugacy class} $C$ in $W$, we have  $w\to_F w'$ and  $w'\to_F w$, cf. \cite[Corollary 4.4]{HeNieminlength}. The definition of an elliptic $F$-conjugacy class, which we do not recall here, can be found in \textsection 4.1 of loc. cit.

\section{Cohomology of compactified Deligne--Lusztig varieties}\label{seccohofcompactified}

\subsection{Main theorem and applications}

\begin{theorem}\label{mainthm}
For each $w\in S^*$, there exists 
an $F$-conjugacy class $C$ in $W$ and a minimal length element $x_0\in C$ such that for any  reduced expression $x\in S^*$ of $x_0$, there are $G^F$-equivariant isomorphisms
	\begin{equation*}
		H^i\Big(\XF{w},\Osh_{\XF{w}}\Big)\overset{\sim}{\longrightarrow} H^i\Big(\XF{x},\Osh_{\XF{x}}\Big)
	\end{equation*}
	and
	\begin{equation*}
		H^i\Big(\XF{w},\omega_{\XF{w}}\Big)\overset{\sim}{\longrightarrow} H^i\Big(\XF{x},\omega_{\XF{x}}\Big),
	\end{equation*}
 for all $i\geq 0$.

If $C$ is moreover an elliptic conjugacy class in $W$, then 
 the isomorphisms above hold if we replace $x_0$ by any other minimal length element in $C$.
		\end{theorem}

\begin{proof}
We prove the theorem for the structure sheaf. The results for the canonical sheaf follows from Serre duality for Deligne--Mumford stacks \cite[Corollary 2.10]{Nironi}.

Let  $w\in S^*$. We showed in Proposition \ref{propcohwrtbraidrelations} that the cohomology for the structure sheaf of $\XF{w}$ and $\XF{w'}$ are isomorphic for any other lift $w'$ of $\beta(w)$ under the projection $\beta:S^*\twoheadrightarrow B^+$.

Using the description of the submonoid of reduced elements in $B^+$ given by Matsumoto's theorem \cite[Theorem 1.2.2, Exercise 4.1]{GeckPfeiffer}, we know that the criterion for $u\in B^+$ to come from a reduced expression in the Weyl group is that $u$ has no subexpression of the form $ss$ for any $s\in S$.

If $w=w_1ssw_2$, 
we apply Lemma \ref{lemswFstowFsP1} and Proposition \ref{propcohwrtP1fibration} to obtain $w_1sw_2$ such that $\XF{w}$ and $\XF{w_1sw_2}$ have the same cohomology for the structure sheaf. Inductively, this yields an element
 $v\in S^*$ with $\ellm{v}\leq \ellm{w}$ such that $v$ has no subexpression of the form $ss$ for any $s\in S$ and for all $i\geq 0$, 
\begin{equation*}
	H^i\Big(\XF{v},\Osh_{\XF{v}}\Big)\overset{\sim}{\longrightarrow}H^i\Big(\XF{w}, \Osh_{\XF{w}}\Big).
\end{equation*}
In other words, we may assume $v\in S^*$ gives a reduced expression of $\alpha(v)$ in $W$. 

By \cite[Theorem 2.6]{GeckKimPfeifferMinlength}, there exists $v'\in W$ which is of minimal length in its $F$-conjugacy class such that $\alpha(v)\to_F v'$ and $\ell(v')\leq \ell(\alpha(v))$. Suppose $\alpha(v)$ is  a minimal length element in its $F$-conjugacy class. By Matsumoto's theorem \cite[Theorem 1.2.2]{GeckPfeiffer}, 
there exists a lift $y\in B^+$ of $\alpha(v)$ such that any reduced expression $x\in S^*$ of $\alpha(v)$ satisfies $\beta(x)=y$. The statement of the theorem thus follows from Proposition \ref{propcohwrtbraidrelations}. 

In general, $\alpha(v)\to_F v'$ gives $v_1',\ldots,v_m'\in W, s_1,\ldots s_{m-1}\in S$ such that $v_1'=\alpha(v), v_m'=v'$  and $v_j'=s_{j-1}v_{j-1}'F(s_{j-1})$ for all $j=2,\ldots,m$. 
Let $a\in \{2,\ldots, m\}$ be the smallest integer such that $\ell(v_a')<\ell(v_{a-1}')$. For all $j=1,\ldots ,a$, let $v_j\in S^*$ be the canonical lift of the corresponding reduced expression of $v_j'$. 
Up to taking different lifts of braid relations, for each $j=2,\ldots ,a-1$, we have $v_j=s_{j-1}z_{j-1}$ (resp. $v_j=z_{j-1}F(s_{j-1})$) and $v_{j-1}=z_{j-1}F(s_{j-1})$ (resp. $v_{j-1}=s_{j-1}z_{j-1}$) for some $z_{j-1}\in S^*$ (cf. \cite[Remark 2.3]{GeckKimPfeifferMinlength}).
Thus it follows from Corollary \ref{corcohisomcyclicshiftingoperator} that $\XF{v_j}$ all have the same cohomology for the structure sheaf as $\XF{v}$ for $j=2,\ldots , a-1$. 

On the other hand, we have $v_{a-1}=s_{a-1}v_aF(s_{a-1})$ in $S^*$. 
We know that $\XF{v}$ has the same cohomology as $\XF{s_{a-1}v_a}$ and $\XF{v_aF(s_{a-1})}$ by  Proposition
\ref{propcohwrtP1fibration}. In this case, we reset $v$ to be $s_{a-1}v_a$. This process will terminate and yields an element $x_0\in W$ which is of minimal length in its $F$-conjugacy class, cf. \cite[Theorem 2.6]{GeckKimPfeifferMinlength}.

It remains to consider the case when $C$ is an elliptic conjugacy class in $W$. By \cite[Corollary 4.4]{HeNieminlength}, we have $x\to_F y$ for any two minimal length elements $x,y\in C$.
The corresponding
 cohomology of $\XF{x}$ and $\XF{y}$ for the structure sheaf are isomorphic by  Corollary \ref{corcohisomcyclicshiftingoperator}.
\end{proof}

\begin{remark}
	We use the same notations as in \Cref{mainthm} and set $I={\supp}_F(x)$. There are $G^F$-equivariant morphisms  given by \cite[Proposition 2.3.8]{DMR}:
\begin{equation*}
	\XF{x}\overset{\sim}{\longrightarrow}G^F/U_I^F\times^{L_I^F} {\overline{X}^{S^*}_{L_I}(x)} \overset{\pi_I}{\longrightarrow} G^F/P_I^F,
\end{equation*}
where   $P_I=BW_IB$ is the parabolic subgroup with unipotent radical $U_I$  satisfying $P_I=L_I\ltimes U_I$ for a Levi subgroup $L_I$ of $G$. Since $G^F$ acts on the fibres of $\pi_I$ transitively, we have a $G^F$-equivariant isomorphism
	\begin{equation*}
		H^i\Big(\XF{x},\Osh_{\XF{x}}\Big)\overset{\sim}{\longrightarrow} {\rm Ind}_{P_I^F}^{G^F}H^i\Big(\overline{X}^{S^*}_{L_I}(x),\Osh_{\overline{X}^{S^*}_{L_I}(x)}\Big).
	\end{equation*}
for all $i\geq 0$. 

\end{remark}

\subsection{Suzuki and Ree groups}\label{secsuzukiree}

Here we recall the assumptions needed on $d$ and $q$ for $(G,F)$ of type $^{2}B_2$, $^{2}F_4$, and $^{2}G_2$.

In addition to the assumptions \Cref{secbackground}, let $d=2$ and let $q$ be an odd power of $\sqrt{2}$ (resp. $\sqrt{2}$, $\sqrt{3}$).
Let $G$ be of type $B_2$ (resp.  $F_4$, $G_2$). Then the endomorphism $F:G\to G$  induces an automorphism of $W$ which preserves $S$ 
 and  $F^2$ is the $q^2$-Frobenius \cite[\textsection 4.3.7]{DigneMichelBook}.
 The group of fixed points $G^{F}$ is the Suzuki group  $^{2}B_2(q)$ (resp. Ree groups $^{2}F_4(q)$, and $^{2}G_2(q)$).

\subsection*{}\textit{Acknowledgements}

The author thanks Jochen Heinloth and Georg Linden for their comments on an earlier draft of this paper. 
Thanks to Ulrich G\"{o}rtz for helpful conversations.  

Parts of the work are carried out when the author was a  member of the trimester program "The Arithmetic of the Langlands Program" at
the Hausdorff Research Institute for Mathematics in Bonn  and when the author was a member of the Research Training Group 2553 in Essen, both funded by the DFG (Deutsche Forschungsgemeinschaft).

	\printbibliography

	\textsc{Fakult\"{a}t f\"{u}r Mathematik, Universit\"{a}t Duisburg-Essen, 45127 Essen, Germany}

  \textit{E-mail address:}
\email{yingying.wang@uni-due.de}
	
\end{document}